\documentclass[11pt]{amsart}

\usepackage{amssymb}
 \usepackage{latexsym}
\usepackage{amsmath}
\newtheorem{theorem}{Theorem}[section] % 1st argument is your name for it
\newtheorem{corollary}[theorem]{Corollary}
\newtheorem{proposition}[theorem]{Proposition}

\theoremstyle{definition}
\newtheorem{example}[theorem]{Example}
 
  \newcommand{\norm}{\left\Vert\,\cdot\,\right\Vert}

\newcommand {\C}{\mathbb C}

\newcommand {\N}{\mathbb N}

\newcommand {\I}{\mathbb I}

\newcommand {\s}{\smallskip}
\newcommand {\m}{\medskip}
\newcommand{\lv}{\left\vert}
\newcommand{\rv}{\right\vert}
\newcommand{\lV}{\left\Vert}
\newcommand{\rV}{\right\Vert}

\newcommand{\dd}{\,{\rm d}}

\newcommand{\CBA}{commutative Banach algebra}

\begin{document}

%\batchmode

\title{Maximal ideals in commutative Banach algebras} 
\author{H.\ G.\ Dales}

\maketitle

\begin{abstract}
 {We show that each maximal ideal in a commutative Banach algebra has codimension $1$.}
\end{abstract}

\section{}
\noindent Let $A$ be a commutative algebra (over the complex field, $\C$). An ideal $M$ in $A$ is {\it maximal\/} if $M \neq A$ and if $M=I$ or $I=A$ 
for each ideal $I$ in $A$ with $M\subset I$.\s

We wonder  if every maximal ideal in a {\CBA}  is of codimension $1$ in $A$?\s

This is not resolved in \cite{D}; I cannot see a comment on the point  in other texts.  Here we show that the above is indeed the case.

Let $A$ be an algebra.  Then $A^{[2]} = \{ab: a,b\in A\}$ and $A^2$ is the linear span of $A^{[2]}$, so that $A^2$ is an ideal in $A$.

Let $A$ be a  commutative algebra without an identity. We denote by $A^\sharp$ the unital  algebra formed by adjoining an identity, $e$, to $A$,
 so that $A$ is a maximal ideal in $A^\sharp$; in the case where $A$ is a Banach algebra, $A^\sharp$ is also a Banach algebra.
 
In the case where $A$ is a topological algebra, a   maximal ideal is either closed or dense in $A$.

First suppose that $M$ is a maximal modular ideal in a commutative algebra $A$. Then $A/M$ is a field containing $\C$. 	In the case where $A$ is a 
{\CBA}, $M$ is necessarily closed, and so $A/M$ is a Banach algebra.  By the Gel'fand--Mazur theorem $A/M \cong \C$, and $M$  is the kernel
 of a continuous character.  In particular, $M$ has codimension $1$.  In fact, the  Gel'fand--Mazur theorem applies to locally 
convex $F$-algebras \cite[Theorem 2.2.42]{D},  and so the previous comment also applies  to this class of topological algebras.

Now suppose that $A$ is a commutative, unital Fr{\'e}chet algebra.   Then a maximal ideal  (which  is necessarily modular)  is not necessarily
 either closed or of codimension $1$, as the following example, which essentially repeats \cite[Proposition 4.10.27]{D},  shows.\s

\begin{example}{\rm Let  $O(\C)$ denote the  set of entire functions on $\C$. This is a commutative, unital algebra   for the pointwise algebraic 
operations, and it is a Fr{\'e}chet algebra with respect to the topology of uniform convergence on compact subsets of $\C$.  It is standard 
that each maximal modular ideal $M$ of codimension $1$ in $O(\C)$ is closed and is such that there exists $z\in \C$ such that
$$
M = M_z:= \{f \in O(\C): f(z)=0\}\,.
$$

Now let $I$ be the set of functions $f\in O(\C)$ such that $f(n)=0$ for each sufficiently large $n\in\N$. Clearly $I$ is an ideal in $O(\C)$, 
and it is easy to see that $I$ is dense in $O(\C)$. Since $O(\C)$ has an identity, $I$ is contained in a maximal modular ideal, say $M$, of  $O(\C)$.
The ideal $M$ is dense in $O(\C)$. Clearly, $M$ is not of the form $M_z$   for any $z\in \C$, and so $M$ does not have codimension
 $1$ in $O(\C)$; the quotient $A/M$ is a `large field'.}\qed
\end{example}\s

\section{}

\noindent  Now suppose that $A$ is  a commutative algebra, and that $M$ is a maximal ideal in $A$. Set $I =A^2+ M$, so that $I$ is 
an ideal in $A$ containing $M$. Thus  either $A^2 \subset M$ or $A^2+M=A$.

Consider the case in which $A^2 \subset M$, so that $A/M$ is a commutative algebra with zero product.  Let $E$ be a subspace of codimension
  $1$ of $A$ with $E\supset M$.   Then $E$ is an ideal in  $A$, and so $E=M$.   Thus it is indeed the case that $M$ has codimension $1$. 
 (This remark is essentially Exercise 6(e) in \cite[Chapter 1]{G}.)
  
  By considering an infinite-dimensional Banach space with zero product, we see that there are examples of maximal  ideals of
 codimension $1$ that are closed and that are dense in a  {\CBA} $A$.

It remains to consider the case where $A^2+M=A$.\s
  
  \section{}
  
\noindent   We now show that the latter case does not occur in a {\CBA} $A$.\s
  
  \begin{proposition}\label{3.1}
  Let $A$ be a {\CBA}. Suppose that $I$ is an ideal in $A$ such that $I$ is dense in $A$  and  $A^2+I=A$.  Then $I$ is not  maximal.
   \end{proposition}
  
  \begin{proof} Necessarily $A$ does not have an identity. The norm in $A^\sharp$  is  denoted by  $\norm$.

For each $n\in\N$, we denote by  $A_n$ the set of elements $a\in A$ such that there exist $b_1,\dots,b_n,c_1, \dots,c_n\in A$ and $x \in I$ with
\begin{equation}\label{(3.1)}
a = b_1c_1+\cdots+ b_nc_n + x\,.
\end{equation}

Choose an element $a_0\in A \setminus I$ such that  $a_0$ has a representation of the form (\ref{(3.1)}) and such that $n$ is the minimum
 natural number with this property.  Define
$$
J = A^\sharp a_0 +I\,.
$$
Then $J$ is an ideal in $A$ with $a_0\in J$, and $J \supsetneq I$.  We shall show that $J\neq A$, and hence that $I$ is not a maximal ideal in $A$.

Assume towards a contradiction that $J=A$. Then $b_1\in J$, and so there exists $d_1 \in  A^\sharp$ such that 
$$
b_1- d_1a_0 \in I\,,
$$
say $\lV d_1 \rV = m$.  Choose $d_2 \in A$ such that $c_1 -d_2 \in I$ and $\lV d_2\rV < 1/m$; this is possible because $I$ is dense in $A$.

Now we have 
$$
b_1 -d_1(b_1d_2+b_2c_2+\cdots+ b_nc_n) \in I\,,
$$
  and so 
$$
b_1(e - d_1d_2) \in d_1(b_2c_2+\cdots+ b_nc_n) + I\,.
$$ 
 However $\lV d_1d_2\rV \leq \lV d_1\rV\,\lV d_2\rV < 1$, and so the element $e - d_1d_2$ is invertible in $A^\sharp$, say with inverse $d_3$. Thus
\begin{eqnarray*}
 a_0 &\in &  c_1d_1d_3(b_2c_2+\cdots+ b_nc_n) + (b_2c_2+\cdots+ b_nc_n) + I\\
 &=& (e+   c_1d_1d_3 )(b_2c_2+\cdots+ b_nc_n) +I\,.
\end{eqnarray*}
In the case where $n>1$, this shows that $a_0\in A_{n-1}$, a contradiction of the minimality of the choice of $n$. In the case where $n=1$,
  the argument shows that $b_1(e - d_1d_2) \in I$ and hence that $b_1 \in I$, and then that $a_0 =b_1c_1 \in I$, a contradiction of the fact that 
  $a_0\notin I$.  Thus $J\subsetneq A$.
 \end{proof}\s
 
 There is a closely-related algebraic result that is surely known and in some text.\s
 
   \begin{proposition}\label{3.1a}
  Let $R$ be a commutative, radical algebra. Suppose that $I$ is an ideal in $R$ such that $R^2+I=A$.  Then $I$ is not  maximal.
   \end{proposition}
   
     \begin{proof}  The element $c_1d_1$  in the above proof (with $R$ for $A$) is such that $e -c_1d_1$ is invertible in  $R^\sharp$ 
because $R$ is radical, and so the argument of the above proof gives the result.
     \end{proof}\s

\begin{theorem} \label{3.3} Let $A$ be a {\CBA}. Then every maximal ideal $M$  in $A$ has codimension $1$ in $A$.  Further, either
 $A/M\cong \C$ or $A^2\subset M$.
\end{theorem}

\begin{proof}   Let $M$ be a maximal ideal  in $A$. We have noted that  $M$ is either closed or dense in $A$.
 We have also noted that either $A^2 \subset M$ or $A^2+M=A$, and that in the former case $M$ does have codimension $1$
 in $A$.
 
 Thus we may suppose that $A^2+M=A$. By Proposition \ref{3.1}, it cannot be that $M$ is dense in $A$, and so $M$ is closed.  Thus  
 $A/M $ has dimension $1$, and so $M$ has codimension $1$ in $A$.  Further, $A/M\cong \C$.
\end{proof}\s

\begin{corollary} \label{3.4}
Let $R$ be a commutative, radical Banach algebra such that $R^2=R$.  Then $R$ has no maximal ideals.
\end{corollary}

\begin{proof}  Assume that $M$ is a maximal ideal in $R$. Then it is  not the case that $R/M\cong \C$ because $R$ is radical, and it is  not the case that
$R^2\subset M$ because $R^2=R$.  In either case, we have a contradiction of Theorem \ref{3.3}.  Thus $R$ has no maximal ideals.
\end{proof}\s

There are many examples of  commutative, radical Banach algebras such that $R^2=R$. Each commutative, radical Banach algebra with a
 bounded approximate identity has this property.  For example, this is the case for the Volterra algebra $\mathcal V$,
which is the space $L^1[0,1]$  with convolution product $\star$ given by
$$
(f\,\star\,g) (t) = \int_0^1 f(t-s)g(s){\dd}s\quad (t \in [0,1])
$$
for $f,g \in \mathcal V$.  There are also examples which are integral domains.

The following example shows that there are commutative, radical Banach algebras $R$ such that $\overline{R^2}=R$,
 but such that $R$ does have a maximal ideal, necessarily of codimension $1$.\s
 
 \begin{example}{\rm  Let $\I =[0,1]$, and define 
$$
R = \{f \in C(\I) : f(0)=0\}\,,
$$
taken with the uniform norm $\lv \,\cdot\,\rv_\I $ and the above truncated convolution product.  Then $R$ is a commutative, radical Banach algebra.

Set $I = \{f \in R: \lim_{t\to 0+}f(t)/t =0\}$, so that $I$  is a linear subspace of $R$; in fact, $I$ is an ideal in $(R,\,\star\,)$.

Take $f,g \in R$ and $\varepsilon > 0$.  Then there exists $\delta>0$ such that $$\lv g(s)\rv_\I < \varepsilon\quad(0\leq s\leq \delta)\,.$$
Thus, for $0\leq t\leq \delta$, we have
$$
\lv (f\,\star\,g) (t)\rv \leq \lv f\rv_\I\int_0^t \lv g(s\rv{\dd}s < \varepsilon t \lv f\rv_\I\,,
$$
and so $f\,\star\,g \in I$. Hence $R^2\subset I \subset M$, and so $M$ is a  maximal ideal (of codimension $1$) in $R$.
}\qed
\end{example}\s
 
 \m

\noindent Department of  Mathematics and Statistics\\
  University of Lancaster\\
Lancaster LA1 4YF\\
United Kingdom \\
 {g.dales@lancaster.ac.uk}
\end{document}